\documentclass{amsart}
\usepackage{amsmath, amsthm, amsfonts, amssymb}
\usepackage{mathrsfs,graphicx}
\usepackage{mathtools}
\usepackage{bbm}
\usepackage{dsfont}
\usepackage{ifthen}
\usepackage{a4}
\usepackage{hyperref}
\usepackage{enumerate}

\numberwithin{equation}{section}
\newtheorem{thm}{Theorem}[section]
\newtheorem{cor}[thm]{Corollary}

\newtheorem{lemma}[thm]{Lemma}

\theoremstyle{remark}
\newtheorem{rem}[thm]{Remark}

\theoremstyle{definition}

\newcommand{\applied}[2]{\langle #1,#2\rangle}

\DeclarePairedDelimiter\norm{\lVert}{\rVert}
\DeclarePairedDelimiter\abs{\lvert}{\rvert}

\renewcommand{\phi}{\varphi}

\newcommand{\eps}{\varepsilon}

\newcommand{\R}{\mathds{R}}
\newcommand{\C}{\mathds{C}}

\newcommand{\N}{\mathds{N}}

\newcommand{\Q}{\mathds{Q}}

\newcommand{\cM}{\mathscr{M}}

\newcommand{\cF}{\mathscr{F}}

\begin{document}

\title{A Generalization of L\'evy's Theorem on positive matrix semigroups}
\author{Moritz Gerlach}
\address{Moritz Gerlach, Universit\"at Potsdam, Institut f\"ur Mathematik, Karl--Liebknecht--Stra{\ss}e 24–25, 14476 Potsdam, Germany}
\email{gerlach@math.uni-potsdam.de}

\begin{abstract}
	We generalize a fundamental theorem on positive matrix semigroups stating that each component is either strictly positive for all times or identically zero (``L\'evy's Theorem'').
	Our proof of this fact that does not require the matrices to be continuous at time zero.
	We also provide a formulation of this theorem in the terminology of positive operator semigroups on sequence spaces.
\end{abstract}

\date{\today}
\dedicatory{To Caro}

\maketitle

\section{Introduction}

Let $I$ denote a countable set. A family of functions $p_{ij}\colon (0,\infty) \to [0,\infty)$, $i,j \in I$, 
is called a \emph{positive matrix semigroup (or transition semigroup)} over the index set $I$ if for every $i,j \in I$ the following two conditions are fulfilled:
	\begin{enumerate}[(a)]
		\item\label{item:transitioncontinuous} $p_{ij}$ is continuous on $(0,\infty)$ and
		\item $p_{ij}(t+s) = \sum_{k\in I} p_{ik}(t)p_{kj}(s)$ for all $t,s >0$,
	\end{enumerate}

The main result of the present article is the following theorem (see Theorem \ref{thm:oncepositivealwayspositive}):

\begin{thm}
\label{thm:main}
	Let $p_{ij}$ be a positive matrix semigroup over a countable index set $I$. Then for all  $i,j \in I$
	\begin{align}
	\label{eqn:statement}
	\text{either }p_{ij}(t)=0\text{ for all }t>0\text{ or }p_{ij}(t)>0\text{ for all }t>0.
	\end{align}
\end{thm}

We also transfer this theorem to the theory of operator semigroups, where it becomes the following (see Corollary  \ref{cor:semigroups}):

\begin{thm}
\label{thm:semigroups}
Let $E$ be a sequence space (as defined in Section \ref{sec:semigroups}).
Let $(T_t)_{t \in (0,\infty)}$ be a positive one-parameter semigroup of sequentially order continuous operators on $E$ such that $t\mapsto (T_t e^i)_j$ is continuous on $(0,\infty)$ for all $i,j \in \N$.

Then for all $i,j \in\N$ either $(T_t e^i)_j>0$ for all $t>0$ or $(T_t e^i)_j=0$ for all $t>0$.
\end{thm}

Here, $e^i \coloneqq \big(\mathds{1}_{\{i\}}(j)\big)_{j\in\N}$ denotes the \emph{$i$-th unit vector}.
On sequence spaces with a complete and order continuous lattice norm, such as $c_0$ (space of null sequences)  or $\ell^p$ (space of $p$-summable sequences for $1\leq p < \infty$),
the result simplifies as follows (see Corollary \ref{cor:measurablesemigroups}):

\begin{cor}
Let $E$ denote a sequence space with a complete and order continuous lattice norm and
let $(T_t)_{t\in (0,\infty)}$  be a positive one-parameter semigroup such that $t\mapsto (T_te^i)_j$ is measurable $(0,\infty)$ for all $i,j \in \N$.
	Then, for every $i,j \in \N$, either $(T_te^i)_j>0$ for all $t>0$ or $(T_te^i)_j=0$ for all $t>0$.
\end{cor}

A version for spaces with uncountably many atoms, such as $\ell^p(\R)$ for $1\leq p < \infty$, can be found as Corollary \ref{cor:uncountable}.

\medskip

Under additional assumptions, Theorem \ref{thm:main} is well-known in the literature.
A first partial proof was given by L\'evy in \cite[Thm II.8.1]{levy1951}
for positive matrix semigroups which are \emph{Markovian}, i.e.\
	\begin{enumerate}[(a)]
	\item[(c)]\label{it:markov} $\sum_{j\in I} p_{ij}(t) = 1$ for all  $t>0$ and $i\in I$,
	\end{enumerate}
and the additional assumption that all $p_{jj}$ have finite derivatives at time $0$.  This is why Theorem \ref{thm:main} is sometimes referred to as ``Levy's Theorem''.
After Chung had pointed out some inaccuracies to him, L\'evy amended his arguments for the general Markovian case in \cite[Sec 7]{levy1958}.

According to the notes in \cite{chung1960}, an alternative proof for Markovian matrix semigroups was found independently by D.\ G.\ Austin, simplified by Chung using probabilistic methods and published as \cite[Thm II.5.2]{chung1960}.
Another algebraic proof was found by D.\ Ornstein and published as \cite[Thm II.1.5]{chung1960}. Both proofs make use of the fact that for a Markovian matrix semigroup
for all $i,j$ the limits $\lim_{t\downarrow 0}p_{ij}(t)$ exist, cf.\ \cite[Thm II.1.3]{chung1960}.

In \cite[Thm 10.1]{chung1963}, Chung proved Statement \eqref{eqn:statement} in the most general situation so far, where the Markov condition (c) is replaced by 
continuity at $t=0$, i.e.\ 
	\begin{enumerate}[(a)]
	\item[(d)]\label{it:standard} $\lim_{t\to 0} p_{ij}(t) = \begin{cases} 0 & i\neq j \\ 1 & i=j \end{cases}$ for all $i,j \in I$.
	\end{enumerate}
Matrix semigroups satisfying property (d) are also called \emph{standard} in the literature.
Unfortunately, the proof given in \cite[Thm 10.1]{chung1963} is apparently not complete: it is unclear how to conclude $p(2t_0)\to 0$ as $N\to 0$ from the uniform convergence of (10.2) at the top of page 76, as several objects
depend on $N$. This has recently led to some uncertainty regarding correctness of the theorem within the operator semigroup community.
However, it is not difficult to understand the author's intention by taking his previous publications as \cite[Thm II.1.5]{chung1967} into account.

\medskip

In the theory of operator semigroups a version of Theorem \ref{thm:semigroups} for positive analytic $C_0$-semigroups on general (not necessarily discrete) 
Banach lattices can be found in \cite[Thm C-III 3.2(b)]{nagel1986}. 
For self-adjoint $C_0$-semigroups on $L^2$-spaces this result is due to Simon \cite[Thm~1]{Simon1973}.
For $C_0$-semigroups on $L^p$-spaces it was shown by Kishimoto and Robinson \cite[Lem~8]{KishimotoRobinson1981}
and it was generalized to the Banach lattice setting by Majewski and Robinson \cite[Prop~2]{MajewskiRobinson1983}.
Jochen Gl\"uck pointed out to the author that, similar to the main result of the present article, the assumption of strong continuity at time zero can also be dropped in the analytic case.
This result is provided as Theorem \ref{thm:analytic} in the appendix.

In the 21th century, results similar to Theorem \ref{thm:semigroups} have also been proven in the context of infinite graphs, cf.\ \cite[Thm 7.3]{haeseler2012} and \cite[Thm 1.26]{keller2021}, which is a special
case of both, the discrete and the analytic setting.

\medskip 
Summarizing, the present article has three goals: First, to give a complete proof of Chung's result.
Second, to generalize it by dropping both assumptions (c) and (d).
And third, to transfer the theorem to the theory of operator semigroups, where it is apparently little known.

\section{Positive Matrix Semigroups}

The proof of the following theorem is a combination and adaption of \cite[Thm II.1.5]{chung1967} and \cite[Thm 10.1]{chung1963}.
\begin{thm}
\label{thm:chung}
	Let $p_{ij}$ be a positive matrix semigroup over a countable index set $I$. 
	Let $i,j \in I$ and $t_0>0$ such that $p_{ij}(t)=0$ for all $0< t \leq t_0$.
	Then $p_{ij}(t)=0$ for all $t > 0$.
\end{thm}
\begin{proof}
	Aiming for a contradiction, first assume that $c\coloneqq p_{ij}(2t_0)>0$.
	Fix $0< \eps < c/7$ and define $\delta\coloneqq 1/16$. By Dini's theorem, the series
	\[ c = p_{ij}(2t_0) = \sum_{k\in I} p_{ik}(t) p_{kj}(2t_0-t)\]
	converges uniformly on $[\delta 2t_0,(1-\delta)2t_0]$. Hence, there exists a finite set $F\subseteq I$ such that 
	\[ \sum_{k\in I\setminus F} p_{ik}(t) p_{kj}(2t_0-t) < \eps \text{ for all }t\in[\delta 2t_0,(1-\delta)2t_0].\]
	Let $N\coloneqq \#F$ denote the number of elements in $F$.  
	By enlarging the set $F$ if necessary we may assume that $N$ is an even number.
	We define the step size $s\coloneqq \frac{t_0}{4N}$ and the abbreviations $N_0 \coloneqq 8\delta N$ and $N_1 \coloneqq 8(1-\delta)N$.
	Note that $N_0, N_1\in \N$, $N_0s=\delta 2t_0$, $N_1s=(1-\delta) 2t_0$ and $(N_0+N_1)s=2t_0$.

	Next, for all $m\in\N$ we define the sets $C_m \coloneqq \{ k \in I : p_{ik} (ls) = 0 \text{ for all }1\leq l \leq m \}$.
	Since $C_{m+1}\subseteq C_m$ for all $m\in\N$, the sets $D_m \coloneqq C_{m}\setminus C_{m+1}$ are pairwise disjoint.
	While $j\in C_m$ for all $1\leq m\leq 4N$, other sets $C_m$ are possibly empty.  For $m,n\in \N$, $n<8N$, we define
	\[ u(m,n) \coloneqq \sum_{k\in C_m} p_{ik}(ns)p_{kj}(8Ns-ns)
	\text{ and } v(m,n) \coloneqq \sum_{k\in D_m} p_{ik}(ns)p_{kj}(8Ns-ns).\]
	By definition of $C_m$, $u(m+1,n)\leq u(m,n)$ for all $m,n$ and $u(m,n)=0$ whenever $n\leq m$.
	Moreover, $v(m,n)<\eps$ whenever $D_m$ is disjoint from $F$ and $n\in \{ N_0,\dots, N_1\}$.
	
	If $M,m \in \N$ with $m<M$, $k\in C_M$ and $l\in I$ are such that $p_{lk}(ms)>0$, then $l \in C_{M-m}$. In fact, otherwise
	there is $M'\in \N$ with $m<M'\leq M$ such that $p_{il}((M'-m)s)>0$; this would imply that 
	\[ p_{ik}(M's)\geq p_{il}\big((M'-m)s\big)p_{lk}(ms) >0\]
	and thus $k\not \in C_M$.
	From this, we now conclude the following two assertions:
	First, 
	\begin{align*}
	u(m,n+1)-v(m,n+1)&= \sum_{k\in C_{m+1}} \biggl( \sum_{l\in I} p_{il}(ns)p_{lk}(s)\biggr) p_{kj}(8Ns-ns-s) \\
	&= \sum_{k\in C_{m+1}} \biggl( \sum_{l\in C_m} p_{il}(ns)p_{lk}(s)\biggr) p_{kj}(8Ns-ns-s) \\
	&\leq \sum_{l \in C_m} p_{il}(ns) \sum_{k\in I} p_{lk}(s) p_{kj}(8Ns-ns-s) = u(m,n)
	\end{align*}
	for all $m,n \in\N$, $n+1<8N$; and second, as $j\in C_{4N}$,
	\begin{align*}
		c = p_{ij}(2t_0) = \sum_{l\in I} p_{il}(N_1 s)p_{lj}(N_0s)
		= \sum_{l\in C_{4N-N_0}} p_{il}(N_1 s)p_{lj}(N_0s)= u(4N-N_0,N_1).
	\end{align*}
	Combining both observations, we obtain for all $N_0 \leq m \leq 4N-N_0$
	\begin{align}\label{eqn:candvmn} c=u(4N-N_0,N_1) \leq u(m,N_1) = u(m,N_1)-u(m,N_0) \leq \sum_{n=N_0+1}^{N_1} v(m,n). \end{align}

	Recall that at most $N$ of the disjoint sets $D_m$ can have a non-empty intersection with $F$.
	As the set $\{N_0, \dots, 4N-N_0\}$ has $4N-2N_0+1 \geq (4-16\delta)N = 3 N$ elements,
	we find $\cM\subseteq \{ N_0, \dots, 4N-N_0\}$ such that $\#\cM=N$ and $D_m \subseteq I\setminus F$ for all $m\in \cM$.
	Summing over $m \in \cM$, we obtain from \eqref{eqn:candvmn}
	\begin{align*}
	 N c &\leq \sum_{m\in \cM}\sum_{n=N_0+1}^{N_1} v(m,n) = \sum_{n=N_0+1}^{N_1} \sum_{m\in \cM} v(m,n)\\
	 &\leq \sum_{n=N_0+1}^{N_1} \sum_{k\in I\setminus F} p_{ik}(ns)p_{kj}(8Ns-ns) \\
	 &< \sum_{n=N_0+1}^{N_1} \eps = (N_1-N_0)\eps = (8N-16\delta N) \eps = 7N \eps  < Nc,
	\end{align*}
	a contradiction. This proves that $p_{ij}(2t_0)=0$.

	Applying this argument to every smaller value of $t_0$ shows that $p_{ij}(2t)=0$ for every $0<t\leq t_0$.
	Finally, we obtain by induction that $p_{ij}(t)=0$ for all $t>0$.
\end{proof}

\begin{cor}
	\label{cor:eventuallyzero}
	Let $p_{ij}$ be a positive matrix semigroup over a countable index set $I$. 
	Let $i,j \in I$ and $t_0>0$ such that $p_{ij}(t_0)=0$. Then $p_{ij}(t)=0$ for all $t\geq t_0$
\end{cor}
\begin{proof}
	If $p_{ij}(t)=0$ for all $0<t<t_0$, then the assertion follows from Theorem \ref{thm:chung}. 
	Otherwise, there exists $0<t_1< t_0$ such that $p_{ij}(t_1)>0$. 
	Let $t'_0 \coloneqq  \inf \{ t>t_1 : p_{ij}(t)=0\}$. Since $t'_0 \leq t_0$ and $p_{ij}(t'_0)=0$ due to the continuity of $p_{ij}$, 
	we may assume without loss of generality that $t_0$ equals $t'_0$.  Now we consider
	\[ I_1 \coloneqq \{ k\in I : \text{ there exists }0<\delta<t_0 \text{ such that }p_{ik}(t)>0\text{ for all }t_0-\delta<t<t_0\} \]
	and define $I_0 \coloneqq I \setminus I_1$. Note that for every $k\in I_0$ it follows from the continuity of $p_{ik}$ that $p_{ik}(t_0)=0$.
	Let $k\in I_1$ and let $0<\delta<t_0$ such that $p_{ik}(t)>0$ for all $t_0-\delta<t<t_0$. For every $0<s<\delta$ it 
	follows from 
	\[ 0 = p_{ij}(t_0) \geq p_{ik}(t_0-s)p_{kj}(s)\]
	that $p_{kj}(s)=0$. Consequently, by Theorem \ref{thm:chung}, $p_{kj}(t)=0$ for all $t>0$.
	Combining $p_{ik}(t_0)=0$ for $k\in I_0$ and $p_{kj}(t)=0$ for all $t>0$ and $k\in I_1$, we obtain that
	\[ p_{ij}(t_0+t) = \sum_{k\in I_0} p_{ik}(t_0)p_{kj}(t)+\sum_{k\in I_1} p_{ik}(t_0)p_{kj}(t) = 0 \]
	for all $t>0$. This completes the proof.
\end{proof}

In order to obtain Theorem \ref{thm:main}, it remains to show that $p_{ij}(t_1)>0$ implies that $p_{ij}(t)>0$ for all $t\geq t_1$. If the matrix semigroup is standard, this follows easily from
$p_{ij}(t_1+s) \geq p_{ij}(t_1)\big(p_{jj}(s/n)\big)^n$ for all $s>0$ and $n\in\N$;
in case of Markovian matrices the proof is more complex but still based on the fact that  $\lim_{t\to 0} p_{jj}(t)$ always exists (but might be zero), cf.\ \cite[Thm II.1.4]{chung1967}.
In the general setting of this article this assertion is surprisingly sophisticated to prove.
This is the aim of what follows.

\begin{lemma}
\label{lem:induction}
	Let $p_{ij}$ be a positive matrix semigroup over a countable index set $I$ and let $j \in I$ and assume that $p_{jj}(t)=0$ for all $t>0$.
	\begin{enumerate}[(a)]
	\item Let $t>0$, $i,k \in I$ and $(\alpha_n)\subseteq (0,\infty)$. Then for every $N\in\N$  one has
	\begin{align}
	\label{eqn:hyp1}
		&p_{ik}\biggl(t+ \sum_{n=1}^N \alpha_n\biggr) = \sum_{n=1}^N p_{ij}\biggl( t + \sum_{m=1}^{n-1} \alpha_m \biggr)p_{jk}\biggl(\sum_{m=n}^N \alpha_m\biggr)\\
		&\quad+ \sum_{k_1,\dots,k_N \neq j} p_{ik_1}(t)p_{k_1k_2}(\alpha_1) \dots p_{k_{N-1}k_N}(\alpha_{N-1})p_{k_Nk}(\alpha_N). \nonumber
	\end{align}
	\item Let $i \in I$ and $t_1>0$ such that  $p_{ij}(t_1)>0$. 
	Then $p_{jk}(t)=0$ for all $k\in I$ and $t>0$.
	\end{enumerate}
\end{lemma}
\begin{proof}
	(a) We prove the statement by induction over $N$.  For $N=1$ one has
	\[ p_{ik}(t+\alpha_1) = \sum_{k_1 \in I} p_{ik_1}(t)p_{k_1k}(\alpha_1) = p_{ij}(t)p_{jk}(\alpha_1)+ \sum_{k_1\neq j}p_{ik_1}(t)p_{k_1k}(\alpha_1), \]
	which corresponds to \eqref{eqn:hyp1}. 
	Now assume that \eqref{eqn:hyp1} holds for a given $N\in\N$. Using that $p_{jj}$ is identically zero, this implies for $k=j$ that 
	\begin{align}
		\label{eqn:hyp2}
		\sum_{k_1,\dots,k_N \neq j} p_{ik_1}(t)p_{k_1k_2}(\alpha_1) \dots p_{k_{N-1}k_N}(\alpha_{N-1})p_{k_Nj}(\alpha_N) 
		&= p_{ij}\biggl( t + \sum_{n=1}^N \alpha_n \biggr).
	\end{align}
	Now it follows that
	\begin{align*}
		&p_{ik}\biggl(t+ \sum_{n=1}^{N+1}\alpha_n \biggr) = \sum_{k_{N+1}\in I}  p_{ik_{N+1}}\biggl( t + \sum_{n=1}^N \alpha_n \biggr)p_{k_{N+1}k}(\alpha_{N+1})\\
		&= \sum_{k_{N+1}\in I} \biggl( \sum_{n=1}^N p_{ij}\biggl( t+ \sum_{m=1}^{n-1}\alpha_m\biggr)p_{jk_{N+1}}\biggl(\sum_{m=n}^N \alpha_m\biggr)\\
		&\quad + \sum_{k_1,\dots,k_N \neq j} p_{ik_1}(t)p_{k_1k_2}(\alpha_1) \dots p_{k_{N-1}k_N}(\alpha_{N-1})p_{k_Nk_{N+1}}(\alpha_N) \biggr) p_{k_{N+1}k}(\alpha_{N+1})\\
		&= \sum_{n=1}^N p_{ij}\biggl( t + \sum_{m=1}^{n-1}\alpha_m\biggr)p_{jk}\biggl(\sum_{m=n}^{N+1} \alpha_m\biggr) 
		+p_{ij}\biggl(t + \sum_{m=1}^N\alpha_m\biggr)p_{jk}(\alpha_{N+1})\\
		&\quad + \sum_{k_1,\dots,k_{N+1} \neq j} p_{ik_1}(t)p_{k_1k_2}(\alpha_1) \dots p_{k_Nk_{N+1}}(\alpha_N)p_{k_{N+1}k}(\alpha_{N+1}),
	\end{align*}
	where the second equation follows from induction hypothesis \eqref{eqn:hyp1} for $k=k_{N+1}$ and the third from \eqref{eqn:hyp2} and from $p_{jj}=0$.
	This shows that \eqref{eqn:hyp1} holds for $N+1$. 

(b)	Aiming for a contradiction, assume that $p_{jk}(t_2)>0$ for some $t_2>0$ and $k\in I$. By continuity, there exists $\eps>0$ and $\delta>0$ such that 
	$p_{ij}(t_1+t)\geq \eps$ and $p_{jk}(t_2+t)\geq \eps$ for all $t \in [0,\delta]$. Choose a sequence $(\alpha_n) \subseteq (0,\infty)$ such that $\sum_{n\in\N} \alpha_n < \delta$.
	For every $N\in \N$, it now follows from Part (a) that
		\begin{align*}
	 	\sup_{t \in [0,\delta]} p_{ik}(t_1 + t + t_2) &\geq p_{ik}\biggl(t_1 + \sum_{n=1}^N \alpha_n + t_2\biggr)  \\
		&\geq  \sum_{n=1}^{N+1} p_{ij}\biggl( t_1 + \sum_{m=1}^{n-1} \alpha_m \biggr)p_{jk}\biggl(\sum_{m=n}^N \alpha_m + t_2\biggr)\\
		&\geq \sum_{n=1}^{N+1} \eps^2 = (N+1)\eps^2.
	\end{align*}
	This leads to a contradiction as $N \to \infty$. Thus, $p_{jk}(t)=0$ for all $t>0$. \qedhere
\end{proof}

The assumption $p_{ij}(t_1)>0$ for some $t_1>0$ in Part (b) of Lemma \ref{lem:induction} may seem artificial at first glance.
However, the example \[p_{ij}=\begin{cases}1 & j=k \\ 0 & j\neq k\end{cases} \] of a constant matrix semigroup shows that this assumption is indeed necessary.
The next theorem contains the article's main result.

\begin{thm}
\label{thm:oncepositivealwayspositive}
	Let $p_{ij}$ be a positive matrix semigroup over a countable index set $I$. 
	Let $i,j \in I$ and $t_0>0$ such that $p_{ij}(t_0)=0$. Then $p_{ij}(t)=0$ for all $t>0$.
\end{thm}
\begin{proof}
	We prove that for every $k\in I$ at least one of the functions $p_{ik}$ and $p_{kj}$ is identically zero.
	To this end, let $k\in I$ and suppose that $p_{ik}(v)>0$ for some $v>0$. By Corollary \ref{cor:eventuallyzero},
	$p_{ik}(t)>0$ for all $0<t<v$. Hence, we may assume without loss of generality that $v<t_0$. It then follows from
	\[ 0 = p_{ij}(t_0) \geq p_{ik}(v)p_{kj}(t_0-v) \]
	that $p_{kj}(t_0-v)=0$. Aiming for a contradiction, we suppose that $p_{kj}$ is not identically zero. In this case, by Corollary \ref{cor:eventuallyzero},
	\begin{align} \label{eqn:supposep_kj} 
	u \coloneqq \inf \{t>0 : p_{kj}(t) =0 \}>0
	\end{align}
	and for each $0<s<u$ it follows from
	\[ 0 = p_{kj}(u) \geq p_{kk}(s) p_{kj}(u-s)\]
	and $p_{kj}(u-s)>0$ that $p_{kk}(s)=0$. Now Theorem \ref{thm:chung} implies that $p_{kk}(t)=0$ for all $t>0$.
	Therefore, $p_{kl}(t)=0$ for all $t >0$ and $l\in I$ by Part (b) of Lemma \ref{lem:induction}.
	In particular, $p_{kj}$ is identically zero in contradiction to our assumption.

	Hence, for every $t>0$ one has
	\[ p_{ij}(t) = \sum_{k\in I} p_{ik}(t/2) p_{kj}(t/2) = 0. \]
	This completes the proof.
\end{proof}

\section{Operator Semigroups on Sequence Spaces}
\label{sec:semigroups}

In what follows, we prove a version of Theorem \ref{thm:oncepositivealwayspositive} for one-parameter semigroups on sequence spaces.
A family of linear operators $(T_t)_{t\in (0,\infty)}$ on a vector space is called a \emph{(one-parameter operator) semigroup} if $T_{t+s}=T_tT_s$ for all $t,s>0$.

Semigroups appear naturally to describe solutions of linear autonomous evolution equations.
In many applications the underlying Banach space is a function space and thus exhibits some kind of order structure.
If, in such a situation, a positive initial value of the evolution equation leads to a positive solution, one speaks of a positive semigroup.
We refer to the recent monograph \cite{batkai2017} for an introduction to the theory of positive semigroups.

\medskip

In order to make the notion of ``sequence spaces'' precise, we start with the vector space of all real sequences $\ell^0$, 
which is a vector lattice (or Riesz space) with respect to the component-wise ordering. 
We refer to \cite{luxemburg1971} and \cite{meyer1991} for an introduction to the theory of Riesz spaces and Banach lattices, respectively.
For $k\in\N$, we write \[ e^k\coloneqq \big(\mathds{1}_{\{k\}}(n)\big)_{n\in\N} \]
for the \emph{$k$-th unit vector} in $\ell^0$. 
Now we call a sublattice of $\ell^0$, i.e.\ a linear subspace which is closed under lattice operations, a \emph{sequence space} if it contains all
unit vectors $e^k$, $k\in\N$.
All classical ``sequence spaces'', such as $\ell^p$ ($p$-summable sequences for $1\leq p < \infty$), $\ell^\infty$ (bounded sequences), $c$ (convergent sequences) or $c_0$ (null sequences), 
are sequence spaces in the sense of our definition.
However, we note that $\ell^0$ also has many sublattices which are no sequence spaces. For instance, $C(\R)$ is as a vector lattice isomorphic to 
\[ C(\Q) \coloneqq \{ g \in \ell^\infty(\Q) : \text{ there exists } f\in C(\R) \text{ such that } f|_\Q = g \}, \]
which in turn is a sublattice of $\ell^0$.

Let $E$ be a sequence space. We write $E_+\coloneqq \{ x \in E : x_n \geq 0 \text{ for all }n\in\N \}$ for its positive cone and recall that
a linear operator $T\colon E \to E$ is called \emph{positive} if $TE_+ \subseteq E_+$. A semigroup is said to be positive if it consists of positive operators.
A sequence $(x^k)\subseteq E$ is said to be \emph{order convergent} to $x\in E$ if it is order bounded and $x^k_n \to x_n$ for all $n\in\N$ as $k\to \infty$.
We note that for every $x\in E$, $\sum_{n=1}^N x_n e^n$ is order bounded by $\abs{x}$ and thus order convergent to $x$ as $N\to \infty$.  
A linear operator $T\colon E \to E$ is \emph{sequentially order continuous} if $(Tx^k)$ order converges to $Tx$ whenever $(x^k)$ order converges to $x$ in $E$.


Using this terminology, Theorem \ref{thm:oncepositivealwayspositive} becomes the following.

\begin{cor}
\label{cor:semigroups}
Let $E$ be a sequence space and $(T_t)_{t \in (0,\infty)}$ be a positive semigroup on $E$ of sequentially order continuous operators.
If $t\mapsto (T_t e^i)_j$ is continuous on $(0,\infty)$ for all $i,j \in \N$, then
for all $i,j \in\N$ either $(T_t e^i)_j>0$ for all $t>0$ or $(T_t e^i)_j=0$ for all $t>0$.
\end{cor}
\begin{proof}
Let $t,s >0$ and $i,j \in \N$. Since $T_t$ and $T_s$ are order continuous,
\begin{align*} (T_{t+s} e^i)_j  = \biggl( T_s \sum_{k=1}^\infty (T_t e^i)_k e^k \biggr)_j = \sum_{k=1}^\infty (T_t e^i)_k (T_s e^k)_j.
\end{align*}
This shows that $p_{ij}(t) = (T_t e^i)_j$ defines a positive matrix semigroup and Corollary \ref{cor:semigroups} follows from Theorem \ref{thm:oncepositivealwayspositive}.
\end{proof}

\begin{rem}
\label{rem:inverse}
Although the setting of Corollary \ref{cor:semigroups} seems to be more special,
Theorem  \ref{thm:oncepositivealwayspositive} can  be recovered from Corollary \ref{cor:semigroups} as follows:
Let $p_{ij}$ be a positive matrix semigroup over the index set $\N$.  Then
\[E\coloneqq \{ x\in \ell^0 : \sum\nolimits_{n\in\N} \abs{x_n} p_{nm}(t) < \infty \text{ for all } t>0 \text{ and }m\in \N\}\]
is a sublattice of $\ell^0$ (even an ideal) that contains all unit vectors.
On this space, we define the mappings $T_t x \coloneqq \big(\sum_n x_n p_{nm}(t)\big)_{m \in \N}$.
Using the definition of $E$ and the properties of the positive matrix semigroup $p_{nm}$, it is easy to check that 
$T_t\colon E\to E$ defines a one-parameter semigroup that satisfies all conditions of Corollary \ref{cor:semigroups}. 
Hence, for every $n,m \in \N$, $p_{nm}(t)=(T_t e^n)_m$ is either strictly positive for all $t>0$ or identically zero.
\end{rem}

A norm on a vector lattice satisfying $\norm{x} \leq \norm{y}$ whenever $\abs{x} \leq \abs{y}$ is called a \emph{lattice norm};
a vector lattice endowed with a complete lattice norm is called a \emph{Banach lattice}.
A lattice norm on a sequence space is called \emph{order continuous} if every order convergent sequence converges in norm.
For instance, $c_0$ and $\ell^p$ for $1\leq p < \infty$ are sequence spaces with complete and order continuous lattice norms.
On the other hand, the sup norm on $\ell^\infty$ is a complete lattice norm which is not order continuous.

On a sequence space endowed with a complete and order continuous lattice norm, the continuity conditions on the semigroup in Corollary \ref{cor:semigroups} 
can be formally weakened: as in the case for Markovian matrix semigroups, the measurability of the components already implies their continuity.

\begin{cor}
\label{cor:measurablesemigroups}
	Let $E$ be a sequence space with a complete and order continuous lattice norm and let $(T_t)_{t \in (0,\infty)}$ be a positive semigroup on $E$.
If $t\mapsto (T_t e^i)_j$ is measurable on $(0,\infty)$ for all $i,j \in \N$, then the following assertions hold:
\begin{enumerate}[(a)]
\item For all $x\in E$ the mapping $t\mapsto T_t x$ is continuous on $(0,\infty)$.
\item For all $i,j \in\N$ either $(T_t e^i)_j>0$ for all $t>0$ or $(T_t e^i)_j=0$ for all $t>0$.
\end{enumerate}
\end{cor}
\begin{proof}
	As every positive operator on a complete normed vector lattice is continuous, cf.\ \cite[Prop 1.3.5]{meyer1991}, and the norm is order continuous, every $T_t$ is order continuous.

	Let $x\in E_+$. It follows from the order continuity of all $T_t$ that
	\[ t \mapsto (T_t x)_j = \sup_N \sum_{k=1}^N x_k (T_t e^k)_j \]
	is measurable on $(0,\infty)$ for all $j\in \N$. 
	Now let $\phi \in E'$ be a positive functional. Due to the order continuity of the norm, 
	$\phi$ is order continuous and therefore
	\[ t\mapsto \applied{\phi}{T_tx} = \sup_N \sum_{k=1}^N (T_tx)_k \applied{\phi}{e^k} \]
	is measurable on $(0,\infty)$. Decomposing general $x\in E$ and $\phi \in E'$ in positive and negative parts shows that $t\mapsto T_t$ is weakly measurable.
	As $E$ is separable, $t\mapsto T_t$ is strongly measurable. Now assertion (a) follows from \cite[Thm 10.2.3]{hille1957}.

	Assertion (b) now is an immediate consequence of (a) and Corollary \ref{cor:semigroups}.
\end{proof}

\begin{rem}
	The proof of Part (a) of Corollary \ref{cor:measurablesemigroups} also shows that a positive matrix semigroup $p_{ij}$ which is merely known to be measurable on $(0,\infty)$
	is automatically continuous on $(0,\infty)$ if its associated one-parameter semigroup as constructed in Remark \ref{rem:inverse}
	preserves some sequence space with a complete and order continuous norm such as $c_0$ or $\ell^p$ for  $1\leq p < \infty$.
	This observation also includes the well-known fact that every Markovian matrix semigroup is continuous, cf.\ \cite[Thm II.1]{chung1967}.
\end{rem}

While the proof of Theorem \ref{thm:oncepositivealwayspositive} makes heavy use of the fact that the index set is at most countable, 
the statement holds true even on more general discrete spaces with possibly uncountably many atoms whenever the orbits of the considered semigroup necessarily vanishes on all but countably many atoms.
We show in the following corollary that this is the case for positive and strongly continuous one-parameter semigroups on one of the following spaces: $\ell^p(\R)$ for some $1\leq p < \infty$ or
\[ c_0(\R) \coloneqq \{ f\colon \R \to \R : \text{ for all }\eps>0\text{ the set }\{ r\in \R : \abs{f(r)}>\eps\} \text{ is finite}\}\]
endowed with the sup norm. On both spaces the proof is based on the fact that 
\[ \inf_{F\subseteq \R \text{ finite}} \norm{f \cdot \mathds{1}_{\R\setminus F} }=0 \] for all vectors $f$.

\begin{cor}
\label{cor:uncountable}
	Let either $E=\ell^p(\R)$ for some $1\leq p < \infty$ or $E=c_0(\R)$ 
	and let $(T_t)_{t \in (0,\infty)}$ be a positive semigroup on $E$ such that $t\mapsto T_t f$ is continuous on $(0,\infty)$ for all $f\in E$.
Then the following assertions hold:
\begin{enumerate}[(a)]
\item For each $f\in E$ the set $\{ r \in \R : \abs{T_t f}(r) > 0 \text{ for some } t>0 \}$
is at most countable.
\item For all $r,s \in \R$ either $(T_t e^r)(s)>0$ for all $t>0$ or $(T_t e^r)(s) = 0$ for all $t>0$. As before, $e^r=\mathds{1}_{\{r\}}$ denotes the indicator function of the singleton $\{r\}$.
\end{enumerate}
\end{cor}
\begin{proof}
(a) Let $f\in E$ and let $K\subseteq (0,\infty)$ be a compact interval. 
Let $\delta\geq 0$. We are going to show that the set
\[ R_\delta \coloneqq \{ r \in \R : \abs{T_t f}(r) > \delta \text{ for some } t\in K\} \]
is finite.  Using the notation  $\cF \coloneqq \{ F \subseteq \R : F \text{ finite}\}$ we have
$\inf_{F\in \cF} \norm{(T_tf)\mathds{1}_{\R \setminus F}}=0$ for all $t\in K$. 
Moreover, since for all $F \in \cF$,  $\norm{(T_t f)\mathds{1}_{\R \setminus F}}$ is continuous in $t$, the sets
\[ K_F \coloneqq \{ t \in K : \norm{(T_t f)\mathds{1}_{\R \setminus F}} < \delta/2 \}\]
form an open cover of $K$. By compactness of $K$ and because $K_F$ increases with $F$, we find $F\in \cF$ such that $K \subseteq K_F$. 
In particular, 
$\abs{T_t f}(r) < \delta/2$ for all $r\in \R\setminus F$ and $t\in K$. This implies that $R_\delta \subseteq F$ showing that $R_\delta$ is finite.
Since $\delta>0$ and $K\subseteq (0,\infty)$ were arbitrary, this implies assertion (a).

(b) Let $r,s\in \R$. By Part (a), $R\coloneqq \{ v \in \R : (T_te^r)(v)  > 0 \text{ for some } t>0\}$ is at most countable. 
As each $T_t$ is order continuous, the subspace $\ell^p(R)=\{ f\in \ell^p(\R) : f(v)=0 \text{ for all }v\in \R\setminus R \}$
is invariant under each operator $T_t$. Hence, for $s\in R$ the assertion follows from Part (b) of Corollary \ref{cor:measurablesemigroups} applied to the restriction of
$(T_t)_{t \in (0,\infty)}$ to $\ell^p(R)$; for $r\in \R\setminus R$ one trivially has $(T_t e^r)(s)=0$ for all $t>0$.
\end{proof}

\appendix

\section{Analytic Semigroups} 
In what follows, a one-parameter operator semigroup $(T_t)_{t\in (0,\infty)}$ on a Banach space $X$
is called \emph{analytic} (or \emph{holomorphic}) if there exists $\theta \in (0,\frac{\pi}{2}]$ such the operator valued mapping $t \mapsto T_t$ has an analytic extension 
to $\Sigma_\theta \coloneqq \{ z \in \C\setminus \{0\} : \abs{\arg z } < \theta \}$.
It is worth mentioning that the mapping $\Sigma_\theta \ni z \mapsto T_z$ is analytic if and only if $\Sigma_\theta \ni z \mapsto \applied{T_z x}{\phi}$ is analytic for all $x\in X$ and $\phi \in X'$.
In the literature the definition of analyticity typically requires in addition a boundedness condition in the neighborhood of the origin, cf. \cite[Def 3.7.1]{arendt2001} or \cite[Ch 2]{lunardi1995}, 
which is not needed in this section. As the most important example, we mention that on a Hilbert space every semigroup generated by a self-adjoint operator that is bounded above is analytic, cf.\ \cite[Cor II.4.7]{nagel2000};
as the most important counterexample we mention that the shift semigroup on $L^p(\R)$ is not analytic for any $1 \leq  p < \infty$ (and Theorem \ref{thm:analytic} fails to hold).
For the sake of completeness, we recall that a  semigroup $(T_t)_{t\in (0,\infty)}$ is called a \emph{$C_0$-semigroup} 
if $\lim_{t \downarrow 0} T_tx=x$ for all $x\in X$.  An interesting example of an analytic semigroup that is not $C_0$ is for instance given by heat equation with non-local boundary conditions 
on $L^\infty(\Omega)$, cf.\ \cite[Thm 4.8]{arendt2016}.

\begin{thm}
	\label{thm:analytic}
	Let $E$ be a Banach lattice and let $(T_t)_{t \in (0,\infty)}$ be a positive and analytic semigroup on $E$.
	Let $f \in E_+$ and $\varphi \in E'_+$. 
	Then either $\applied{\phi}{ T_t f } = 0$ for all $t \in (0,\infty)$ or $\applied{\phi}{ T_t f}  > 0$ for all $t \in (0,\infty)$.
\end{thm}

A result in the same vein was proven in \cite[Prop~3.12]{AroraGlueckPreprint} for eventually positive rather than positive semigroups; 
while the semigroups considered there are $C_0$, the fact that eventual positivity suffices indicates that the behavior at time $0$ is not important.
The proof of Thm~\ref{thm:analytic} is an adaptation of the classical proof for $C_0$-semigroups and was communicated to the author by Jochen Gl\"uck.

The argument mainly follows the somewhat simplified presentation of the $C_0$-semigroup proof from \cite[Thm~C-III 3.2(b)]{nagel1986}
-- yet, the argument becomes still a bit more transparent if one splits off the following observation, 
that does not need analyticity of the semigroups, into a separate lemma.

\begin{lemma}
	\label{lem:orbit-below}
	Let $(T_t)_{t \in (0,\infty)}$ be a positive semigroup on a Banach lattice $E$. 
	Let $f \in E_+$ and assume that $T_t f \to f$ for $t \downarrow 0$.
	Then there exists a sequence $t_n \downarrow 0$ 
	and an increasing sequence $(f_n)$ in $E$ that converges to $f$ 
	and that satisfies $0 \le f_n \le T_{t_n} f$ for all $n$.
\end{lemma}

\begin{proof}
	Let $(t_n)$ converge to $0$ sufficiently fast that $\sum_n \norm{f - T_{t_n}f} < \infty$  
	and define
	\begin{align*}
		\tilde f_n \coloneqq  f - \sum_{k=n}^\infty \big(f - T_{t_k}f\big)^+
	\end{align*}
	for each $n$.
	Then $(\tilde f_n)$ is an increasing sequence in $E$ that converges to $f$ 
	and for each $n$ one has $\tilde f_n \le f - (f - T_{t_n}f)^+ \leq T_{t_n}f$. 
	By defining $f_n \coloneqq \tilde f_n^+$ for each $n$ one obtains a sequence $(f_n)$ in $E_+$ with the desired properties.
\end{proof}

\begin{proof}[Proof of Theorem~\ref{thm:analytic}]
	Let $\applied{\phi}{T_s f} = 0$ for some time $s > 0$.
	
	\emph{Step~1:}
	Assume, in addition, that $T_t f \to f$ as $t \downarrow 0$. 
	Then we can choose sequences $(t_n)$ in $(0,s)$ and $(f_n)$ in $E_+$ 
	as in Lemma~\ref{lem:orbit-below}.
	For each fixed index $m$ and all $n \ge m$ we get
	\begin{align*}
		0 \leq T_{s-t_n} f_m  \leq  T_{s-t_n} f_n \leq T_{s}f
	\end{align*}
	so $\applied{ \varphi}{ T_{s-t_n}f_m } = 0$ for all $n \ge m$. 
	Since the sequence $(s-t_n)$ clusters at $s$, the identity theorem for analytic functions 
	implies that $\applied{\varphi}{ T_t f_m } = 0$ for all $t \in (0,\infty)$. 
	By letting $m \to \infty$ we conclude that $\applied{ \phi}{T _t f } = 0$ for all $t \in (0,\infty)$.
	
	\emph{Step~2:} 
	Now let $f$ be general. 
	Due to the analyticity of the semigroup the orbit of $f$ is continuous on $(0,\infty)$, 
	so Step~1 can be applied to the vector $\tilde f := T_{s/2}f$ since $\applied{\phi}{ T_{s/2}\tilde f } = 0$. 
	So $\applied{ \phi}{T_t f} = 0$ for all $t > s/2$ 
	and another application of the identity theorem shows that the same equality remains true for all $t > 0$.
\end{proof}

Step~1 in the proof above is simply the argument for the $C_0$-semigroup case; 
the main point is to observe that this argument does not really need the entire semigroup to be $C_0$ but just the orbit of $f$.

\bibliographystyle{abbrv}
\bibliography{analysis}

\end{document}